\documentclass[12pt]{amsart}

\usepackage[english]{babel}
\usepackage[applemac]{inputenc}
\usepackage[T1]{fontenc}
\usepackage{palatino}
\usepackage{bbm}
\usepackage{float, verbatim}
\usepackage{amsmath}
\usepackage{amssymb}
\usepackage{amsthm}
\usepackage{amsfonts}
\usepackage{graphicx}
\usepackage{mathtools}
\usepackage{enumitem}
\usepackage[pagebackref=false]{hyperref}
\usepackage{color}
\usepackage{tikz}
\usepackage{xcolor}
\usepackage{pgfplots}
\usepackage{caption}
\usepackage{subcaption}
\usepackage{enumerate}
\usepackage{url}

\hypersetup{
	colorlinks=true,
	linkcolor=blue,
	filecolor=magenta,      
	urlcolor=cyan,
	citecolor= red,
}

\newcommand{\ubox}{\overline{\dim}_{\mathrm{B}}}

\newcommand{\lbox}{\underline{\dim}_{\mathrm{B}}}
\newcommand{\boxd}{\dim_{\mathrm{B}}}

\newcommand{\Haus}{\dim_{\mathrm{H}}}

\newtheorem*{thm*}{Theorem}
\newtheorem*{conj*}{Conjecture}

\newtheorem{thm}{Theorem}[section]

\newtheorem{lma}[thm]{Lemma}
\newtheorem{cor}[thm]{Corollary}

\newtheorem{conj}[thm]{Conjecture}
\newtheorem{rem}[thm]{Remark}
\newtheorem{ques}[thm]{Question}

\newcommand{\R}{\mathbb{R}}

\newcommand{\N}{\mathbb{N}}
\newcommand{\hd}{\dim_{\textup{H}}}
\newcommand{\bd}{\dim_{\textup{B}}}
\newcommand{\ubd}{\overline{\dim}_{\textup{B}}}
\newcommand{\lbd}{\underline{\dim}_{\textup{B}}}

\newtheorem{theorem}{Theorem}[section]
\newtheorem{lemma}[theorem]{Lemma}
\theoremstyle{definition}

\theoremstyle{remark}

\numberwithin{equation}{section}
\begin{document}
	\title[Digit expansions of numbers]{Digit expansions of numbers in different bases}
	\author{Stuart A. Burrell}
	\address{Stuart A. Burrell\\
		School of Mathematics \& Statistics\\University of St Andrews\\ St Andrews\\ KY16 9SS\\ UK \\ }
	\curraddr{}
	\email{sb235@st-andrews.ac.uk}
	\thanks{ }
	
\author{Han Yu}
\address{Han Yu\\
	Department of Pure Mathematics and Mathematical Statistics\\University of Cambridge\\CB3 0WB \\ UK }
\curraddr{}
\email{hy351@maths.cam.ac.uk}
\thanks{}
	
	\subjclass[2010]{Primary: 11K55, 28A50, 28A80, 28D05, 37C45.}
	
	\keywords{digit expansion, Graham's problem, Schanuel's conjecture}
	\date{\today}
	\dedicatory{}
	
	\maketitle
	
	\begin{abstract}
		A folklore conjecture in number theory states that the only integers whose expansions in base $3,4$ and $5$ contain solely binary digits are $0, 1$ and $82000$. In this paper, we present the first progress on this conjecture. Furthermore, we investigate the density of the integers containing only binary digits in their base $3$ or $4$ expansion, whereon an exciting transition in behaviour is observed. Our methods shed light on the reasons for this, and relate to several well-known questions, such as Graham's problem and a related conjecture of Pomerance. Finally, we generalise this setting and prove that the set of numbers in $[0, 1]$ who do not contain some digit in their $b$-expansion for all $b \geq 3$ has zero Hausdorff dimension.
	\end{abstract}
	
	\maketitle
	\allowdisplaybreaks
	
	\section{Introduction and Statement of Results}\label{introsec}
	The expansion of numbers in various bases gives rise to a number of notorious problems. Of these, the most famous is the Erd\H{o}s ternary problem \cite{E79}, which conjectures that there are only finitely many integers $n$ such that the ternary expansion of $2^n$ does not contain the digit $1$, see \cite{DW16,L09} for some recent developments. Other important related works include \cite{S80}. In this field, an intriguing folklore conjecture concerning the integer sequence \cite{OEIS2} states the following.
	\begin{conj}\label{IMPOSSIBLE}
		$0,1,82000$ are the only integers whose base $3,4$ and $5$ expansions contain solely the digits $0, 1$.
	\end{conj}
	As well as attracting specialist audiences, this problem has been popularised in \cite{Youtube} and, so far, numerical computations have not found any counter-examples up to $2^{65520}$. To our knowledge, the following is the first progress on this conjecture.
	\begin{thm}\label{A146025}
		For each $\epsilon>0$, there is a constant $C_\epsilon>0$ such that
				\[
		\# \{k \in [1, n] : \text{ the base $4$ and $5$ expansions of $k$ contain only the digits $0, 1$}\} \leq C_\epsilon n^\epsilon.
		\]

	\end{thm}
Of course, the base $3$ requirement may immediately be added to correspond directly with the conjecture.
	\begin{cor}\label{45thm}
		For each $\epsilon>0$, there is a constant $C_\epsilon>0$ such that
		\[
		\# \{k \in [1, n] : \text{ the base $3,4$ and $5$ expansions of $k$ contain only the digits $0, 1$}\} \leq C_\epsilon n^\epsilon.
		\]
	\end{cor}
	Our methods may be used to show other similar results in a range of contexts. For example, one can show that an $O(n^\epsilon)$ estimate holds if we consider the set of numbers whose base $3$ and $7$ expansions contain only binary digits. Numerical computations indicate that the largest such number which is smaller than $7^{3841}$ is between $7^{43}$ and $7^{44}$. Other candidates for the application of our arguments include the integer sequences \cite{OEIS3,OEIS4,OEIS5,OEIS6} from the OEIS \cite{OEIS}.\\

One may also consider those integers whose base $3$ and $4$ expansions contain only binary digits. This corresponds to the integer sequence \cite{OEIS1}, denoted $\mathcal{S}$, defined by
	\[
	\mathcal{S}(n) := \#\{k\in [4^n,4^{n+1}-1]: \text{  the expansions of $k$ contains only $0,1$ in bases $3,4$}     \}.
	\]
	The first few terms of this sequence are $2, 1, 0, 3, 6, 3, 0, 5, 12$ and $11$. Numerical analysis of $\mathcal{S}$, see Figure \ref{fig:figure 1}, suggests that
	\[
	\limsup_{n\to\infty}\frac{\log \mathcal{S}(n)}{n\log 4}=\frac{\log 2}{\log 3}-\frac{1}{2}.
	\]
	In stark contrast to the setting of Theorem \ref{A146025}, there appear to be infinitely many $n$ such that $\mathcal{S}(n) > 0$, and the proof of Theorem \ref{A146025} sheds some light on why the base $5$ requirement induces such a dramatic transition. In addition, one may wonder if there are infinitely many $n$ with $\mathcal{S}(n) =0$, and our next result, Theorem \ref{A230360}, confirms this fact. 
	
	\begin{thm}\label{A230360}
		For each $\epsilon>0$, there is a constant $C_\epsilon>0$ such that
		\[
		\mathcal{S}(n) \leq C_\epsilon 4^{n(\log2/\log3-0.5+\epsilon)}.
		\]
		Moreover,
		\[
		\liminf\limits_{n \rightarrow \infty} \frac{\#\{n\in\mathbb{N}:  \mathcal{S}(n) \cap [3^{n},3^{n+1}]=\emptyset         \}}{n} \geq \log (2.25)/\log 9\approx 0.36907,
		\]
		and if $n$ is such that $9^{\{n\log 4/\log 9\}}\in (1.5,2.25)\cup (4.5,6.75)$, then $n\in \{n\in\mathbb{N}:  \mathcal{S}(n)\cap [3^{n},3^{n+1}]=\emptyset         \}.$
	\end{thm}
From the proof of this theorem it may be deduced that the converse of the last statement does not hold, a fact further discussed in Section \ref{discuss}.\\
	
	\begin{figure}[h]
		\includegraphics[width=0.75\linewidth, height=8cm]{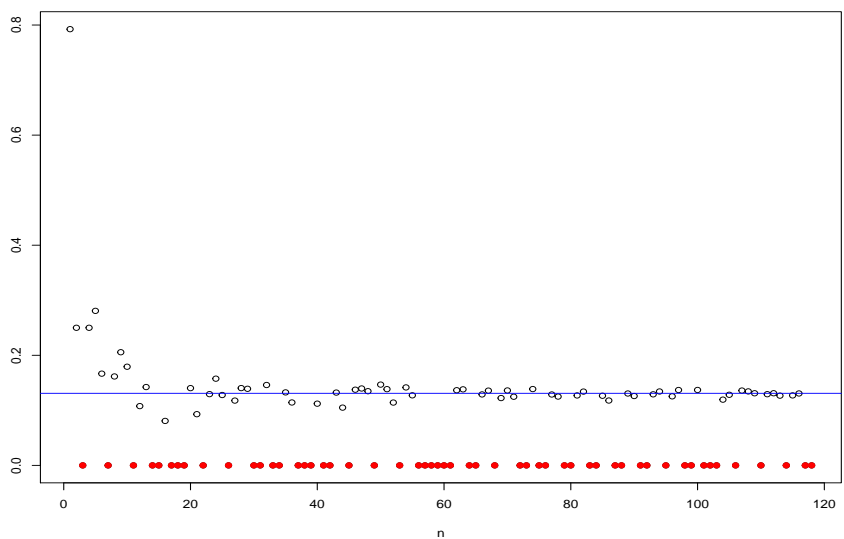} 
		\caption{A plot of $\log \mathcal{S}(n)/\log 4^n$ for $n\in\{1,\dots,118\}$. The horizontal line is $\{y=\log 2/\log 3-0.5\}$ and the lower dots indicate where $\mathcal{S}(n)=0$. According to this data, the density of $0$ in this range is approximately $0.4576271.$}
		\label{fig:figure 1}
	\end{figure}

	The third focus of this paper concerns a generalisation of the above setting, by introducing the notion of \emph{digit-special} numbers. We say real number $x \geq 0$ is {digit-special} if the expansion of $x$ in base  $b$ does not contain at least one digit from $0, 1,\dots, b-1$ for all $b \geq 3$. Our work in this broader direction relies on Schanuel's conjecture \cite{A71}, which we state below for convenience.
	\begin{conj}[Schanuel]\label{Schanuel}
		Let $x_1,\dots,x_n$ be  $\mathbb{Q}$-linearly independent complex numbers, the transcendence degree  of $\mathbb{Q}(x_1,\dots,x_n,e^{x_1},\dots,e^{x_n})$ is at least $n.$ 
	\end{conj}
	Our primary result on digit-special numbers is the following, which, as an initial contribution, we hope will provoke further investigations into this rich topic. 
	\begin{thm}\label{MAIN}
		Assume Schanuel's conjecture. For each $\epsilon>0$, there is a constant $C_\epsilon>0$ such that for all $N\geq 1$, the number of digit-special integers smaller than $N$ is at most $C_\epsilon N^{\epsilon}.$ In addition, the Hausdorff dimension of the set of digit special numbers intersecting $[0,1]$ is zero.
	\end{thm}
	
These results also find connections to a famous question asked by Graham and a related conjecture of Pomerance.
	
	\begin{ques}[Graham's $\$1000$ problem \footnote{According to \cite{OEIS7}, Graham offers $ \$1000$ to the first person with a solution.}]
		Are there infinitely many integers $n\geq 1$ such that the binomial coefficient $\binom{2n}{n}$ is coprime with $105=3\times 5\times 7?$
	\end{ques}
	This problem is currently open but has seen significant attention. Notably, in 1975, Erd\H{o}s, Graham, Ruzsa and Straus showed the following result.
	\begin{thm}[Two prime factor theorem \cite{EGRS75}]\label{EGRS}
		Let $p,q$ be integers greater than $1$. If $A,B$ are two positive integers satisfying
		\[
		\frac{A}{p-1}+\frac{B}{q-1}\geq 1,
		\]
		then there exist infinitely many integers whose base $p$ expansion contains only digits $\leq A$ and base $q$ expansion contains only digits $\leq B.$
	\end{thm}
	Later it will be clear that the following condition seems to be more canonical for this type of problem:
	\[
	\frac{\log (A+1)}{\log p}+\frac{\log (B+1)}{\log q}\geq 1,
	\]
	although we have not yet proved an analogous result with this condition. The connection between prime factors of binomial coefficients and digit expansions is due to Kummer \cite{K52}, who proved that a prime number does not divide $\binom{2n}{n}$ if and only if the $p$-ary expansion of $n$ contains only digits less than or equal to $(p-1)/2$. Thus by the two prime factor theorem of Erd\H{o}s, Graham, Ruzsa and Straus we see that for any two different odd prime numbers $p,q,$ there are infinitely many integers $n$ such that $\binom{2n}{n}$ is coprime with $p$ and $q$.\\
	
	In approaching Graham's problem, it is natural to first consider alternative or related forms. In \cite[Section 4]{P15}, Pomerance gave a heuristic argument that leads to the following conjecture.
	\begin{conj}\label{P}
		Let $N$ be an integer. Denote $G(N)$ to be the number of positive integers $n\leq N$ such that $\binom{2n}{n}$ is coprime with $105$. Then 
		\[
		N^{0.025}\leq G(N)\leq N^{0.026}
		\]
		for all large enough $N.$
	\end{conj}
	Moreover, in \cite[page 639]{P15} Pomerance asks
	\begin{align*}
	\textnormal{``...why would the base-$p$ expansion of $n$ have nothing to do with the base-$q$ }\\
	\textnormal{expansion when $p$ and $q$ are different primes?"}
	\end{align*}
	We partially answer this question and Conjecture \ref{P} in Theorem \ref{PP}, the proof of which may be found in Section \ref{discuss}. For now, this theorem is dependent upon Schanuel's conjecture, due to the important consequence that
	\[
	1, \log 3/\log 5, \log 3/\log 7
	\]
	are then $\mathbb{Q}$-linearly independent.
	
	\begin{thm}\label{PP}
		Let $N$ be an integer and $G(N)$ denote the number of positive integers $n\leq N$ such that $\binom{2n}{n}$ is coprime with any choice of  three different prime numbers. Then we have
		\[
		G(N)\leq N^{0.073}
		\]
		for all large enough $N.$ Furthermore, assuming Schanuel's conjecture, we have
		\[
		G(N)\leq N^{0.026}
		\]
		for all large enough $N.$ 
	\end{thm}
It is likely that one can prove the necessary $\mathbb{Q}$-linear independence directly without having to prove the more general Schanuel's conjecture. We are able to make some progress along these lines and in Section \ref{proof three choices} show the following.
	\begin{thm}\label{three choices}
		The triple
		\[1, \log 3/\log 5, \log 3/\log n\]
		is $\mathbb{Q}$-linearly independent for at least one $n \in \{7, 11, 13\}$.
	\end{thm}
\begin{rem}
	Theorem \ref{three choices} may easily be generalized to show that for each choice of three primes numbers $p_1,p_2,p_3\geq 7,$ at least one of them, say, $n$, is such that
	\[1, \log 3/\log 5, \log 3/\log n\]
	are $\mathbb{Q}$-linearly independent.
\end{rem}
	There are yet more interesting stories in this direction and we postpone further discussion until Section \ref{discuss}.
	
\section{Preliminaries}
	In this section we introduce the required definitions and results from existing literature.
	
	\subsection{Densities of integer sequences}\label{DEN}
	The notion of density describes the size of subsets of $\N$. Let $W\subset\mathbb{N}$ be a sequence of natural numbers and define
	\[
	\#_nW=\#\{i\in [1,n]: i\in W\}.
	\]
	Then, the upper natural density of $W$ is
	\[
	\overline{d}(W)=\limsup_{n\to\infty} \frac{\#_nW}{n},
	\]
	and the lower natural density is given by
	\[
	\underline{d}(W)=\liminf_{n\to\infty} \frac{\#_nW}{n}.
	\]
	If these two numbers coincide we call the common value, denoted $d(W)$, the natural density of $W$.
	\subsection{Dimensions}\label{DIM}
	Dimension is another standard way of quantifying the size of a set. There are numerous notions, but our focus is the Hausdorff and box dimensions. For an in-depth introduction, see \cite[Chapters 2,3]{Fa} and \cite[Chapters 4,5]{Ma1}. 
	
	\subsubsection{Hausdorff dimension}
	
	For all $\delta>0$ and $s > 0$, define the $\delta$-approximate $s$-dimensional Hausdorff measure of a set $F \subseteq \R^n$ by
	\[
	\mathcal{H}^s_\delta(F)=\inf\left\{\sum_{i=1}^{\infty}\mathrm{diam}(U_i)^s: \bigcup_i U_i\supset F, \mathrm{diam}(U_i)\leq \delta\right\},
	\]
	and the $s$-dimensional Hausdorff measure of $F$ by
	\[
	\mathcal{H}^s(F)=\lim_{\delta\to 0} \mathcal{H}^s_{\delta}(F).
	\]
	The Hausdorff dimension of $F$, denoted $\Haus F$, is then given by
	\[
	\Haus F=\inf\{s\geq 0:\mathcal{H}^s(F)=0\}=\sup\{s\geq 0: \mathcal{H}^s(F)=\infty          \}.
	\]
	\subsubsection{Box dimensions}\label{box dimension}
	Let $N(F, r)$ denote the smallest number of cubes of side length $r > 0$ required to cover $F \in \R^n$. The upper box dimension of a bounded set $F$ is
	\[
	\ubox F=\limsup_{r\to 0} \left(-\frac{\log N(F,r)}{\log r}\right),
	\]
	and the lower box dimension of $F$ is
	\[
	\lbox F=\liminf_{r\to 0} \left(-\frac{\log N(F,r)}{\log r}\right).
	\]
	If $\lbd F = \ubd F$, then we call the common value, denoted $\bd F$, the box dimension of $F$. It is easy to see that for all $F \subseteq \R^n$,
	$$
	\hd F \leq \lbd F \leq \ubd F.
	$$
	\subsection{Invariant sets}\label{INV}
	Given an integer $k\geq 2$, let $A_k$ denote an arbitrary closed $\times k\mod 1$ invariant subset of $[0,1]$. That is to say, $a\in A_k$ implies $\{ka\}\in A_k$ for all $a \in A_k$, where $\{x\}$ is the fractional part of $x$. We say that $A_k$ is strictly invariant if $a\in A_k$ if and only if $\{ka\}\in A_k$. For each closed $\times k\mod 1$ invariant set $A_k$, it is known that $\Haus A_k=\ubox A_k$ \cite[Theorem 5.1]{Fu}. In particular, for any integers $k,l\geq 2$, and closed $\times k,l\mod 1$ invariant sets $A_k,A_l$, we have $\Haus A_k\times A_l=\ubox A_k\times A_l.$
	
	\subsection{Equidistribution}\label{Eqularge}
	Let $T$ be a compact metric space and $\mu\in\mathcal{P}(T)$ be a Borel probability measure. Let $X=\{x_n\}_{n\geq 1}$ be a sequence in $T.$ We say that $X$ equidistributes in $T$ with respect to $\mu$ if for each closed metric ball $B\subset T$,
	\[
	\lim_{N\to\infty} \frac{1}{N} \sum_{n=1}^N \mathbbm{1}_{B}(x_n)=\mu(B).
	\]
	Suppose that $X'= \{x_{i_k}\}_{k\geq 1}$ is a subsequence of $X$ such that $\{i_k\}_{k\geq 1}$ has positive upper natural density $\rho > 0$. $X'$ might not be equidistributed, but we may still consider the $\mu$ measure of its closure in $T$ in some special cases. For example, it is not too hard to show that if $T$ is the $n$-dimensional torus and $\mu$ denotes Lebesgue measure, then $\mu(\overline{X'})\geq \rho$.
	
	\subsection{Dipole directions}\label{DD}
	Let $A\subset\mathbb{R}^n$ be compact and consider
	\[
	DD(A)=\{(x-y)/|x-y|:  x,y\in A, |x-y|>0.001\}\subset S^{n-1}.
	\]
	From \cite[Section 4.3]{Y19}, we know that $\ubox A\geq 0.5\ubox DD(A)$. Moreover, if $x\in\mathbb{R}^n$ is a fixed point, then
	\[
	DD(A,x)=\{(x-y)/|x-y|:  y\in A, |x-y|>0.001\}\subset S^{n-1}
	\]
	has dimension $\ubox DD(A,x)\leq \ubox A.$
	%
	
	\subsection{Intersections of invariant sets}\footnote{As a side note, we mention that the result we discuss here actually partially resolves the quoted question of Pomerance in Section \ref{introsec}.}\label{UniSmall}
	The methods we use in this paper rely heavily on the following result from \cite{Y19}, which is a uniform and higher dimensional version of a deep result concerning the Furstenberg problem \cite{Fu2} proven in \cite{Sh} and \cite{Wu}.\\
	
	Let $k\geq 2$ be an integer and $A_{p_1},\dots, A_{p_k}$ be closed invariant subsets of $[0,1]$ with respect to $\times p_1 \mod 1, \times p_2 \mod 1,\dots,\times  p_k \mod 1$, respectively. Assume that $\log p_1/\log p_i$ for $i\in\{2,\dots,k\}$ are irrational numbers which are linearly independent over $\mathbb{Q}$. If $$\sum_{i=1}^k \Haus A_{p_i}<k-1,$$ then for each $2k$-tuple $u_1,\dots, u_k, v_1,\dots,v_k$ of non-zero real numbers we have
	\[
	\ubox \cap_{i=1}^k (u_i A_{p_i}+v_i)=0
	\]  
	by \cite{Y19}. Moreover, for $\delta>0$, if $\delta<|u_i|<\delta^{-1}$ for each $i\in\{1,\dots,k\}$, then for each $\epsilon>0$ there is an integer $N_\epsilon>0$ such that 
	\[
	N(\cap_{i=1}^k (u_i A_{p_i}+v_i),2^{-N})\leq N^{\epsilon}
	\]
for all $N\geq N_\epsilon$, where $N(\dots)$ denotes the box covering number (see Section \ref{box dimension} for details). Note that the choice of $N_\epsilon$ does not depend on $u_i,v_i$. For most of the results in this paper, we do not need the full strength of the above result. In fact, our main results (Theorems \ref{A146025}, \ref{A230360}) only rely on the case $k=2.$ In this case, the above result is \cite[Corollary 1.2]{Y19}. Alternatively, one can apply \cite[Theorem 1.11, Lemma 1.8]{Sh}. For $k\geq 3,$ results in \cite{Sh} cannot be used here. Nonetheless, the result follows by modifying the proof of \cite[Theorem 10.1]{Y19} as described in the discussions found in \cite[Section 12.1]{Y19}.  
	\section{Schanuel's conjecture and proof of Theorem \ref{three choices}}\label{proof three choices}
	In this section, we use Schanuel's conjecture to show $\mathbb{Q}$-linealy independence among ratios of integer logarithms.
	\begin{lma}\label{Sch}
		Assume Schanuel's conjecture. Let $k\geq 3$ be an integer. If $p_1,\dots,p_k$ are integers such that $$p^{n_1}_1\dots p^{n_k}_{k}=1 \,\,\,\,\,\,(n_i \in \mathbb{Z})$$ implies $n_1=\dots=n_k=0$, then 
		\[
		1,\frac{\log p_1}{\log p_2},\dots,\frac{\log p_1}{\log p_k}
		\]
		are $\mathbb{Q}$-linearly independent.
	\end{lma}
	\begin{rem}
		For $k=2$, the conclusion of Lemma \ref{Sch} holds without requiring Schanuel's conjecture.
	\end{rem}
	\begin{proof}
		The required $\mathbb{Q}$-linearly independence follows if
		\[
		\Lambda'=\left(\frac{\prod_{i=1}^{k} \log p_i}{\log p_1},\dots, \frac{\prod_{i=1}^{k} \log p_i}{\log p_k}\right)
		\]
		are $\mathbb{Q}$-linearly independent.
		Considering Conjecture \ref{Schanuel} in the case when $e^{x_1},\dots,e^{x_k}$ are integers, the conjecture reduces to saying that $x_1,\dots,x_k$ are algebraically independent over $\mathbb{Q}.$ We want to apply this conclusion with $x_1=\log p_1,\dots,x_k=\log p_k.$ Now if 
		 if $1,\log p_1,\dots,\log p_k$ are $\mathbb{Q}$-linearly independent, then we meet the conditions of Conjecture \ref{Schanuel} and can apply the aforementioned conclusion. This says that $\log p_1,\dots,\log p_k$ are algebraically independent. Suppose that $\Lambda'$ is not $\mathbb{Q}$-linear independent, then we have
		 \[
		 \sum_{j=1}^k c_j \log p_j=c\prod_{i=1}^k \log p_i
		 \] 
		for some integers $c_1,\dots,c_k$ and $c.$ This contradicts the algebraic independence of $\log p_1,$ $\dots,$ $\log p_k.$ This implies that $\Lambda'$ is indeed $\mathbb{Q}$-linearly independent if $1,\log p_1,\dots,\log p_k$ are $\mathbb{Q}$-linearly independent.
	\end{proof}
	To prove Theorem \ref{three choices}, first recall \cite[Theorem 1, Chapter 2]{L66}. 
	\begin{thm}[Six Exponentials Theorem]\label{Six}
		Let $(x_1,x_2,x_3)$ and $(y_1,y_2)$ be a $\mathbb{Q}$-linearly independent triple and pair of complex numbers, respectively. There exists a pair $(i,j)\in \{1,2,3\}\times \{1,2\}$ such that
		\[
		e^{x_i y_j}
		\]
		is transcendental over $\mathbb{Q}.$
	\end{thm}
\begin{proof}[Proof of Theorem \ref{three choices}]
	Applying Theorem \ref{Six} with \[(x_1,x_2,x_3)=(\log 7/\log 5, \log 11/\log 5,\log 13/\log 5)\] and \[(y_1,y_2)=(\log 3,\log 5),\] we see that at least one of 
	\[
	\exp(\log 3\log 7/\log 5), \exp(\log 3\log 11/\log 5), \exp(\log 3\log 13/\log 5).\]
is not algebraic over $\mathbb{Q}.$ Suppose now that $a,b,c\geq 2$ are  integers, $\log b/\log a\notin\mathbb{Q}$ and $\exp(\log a\log b/\log c)$ is not algebraic over $\mathbb{Q}.$ Then  integer solutions $(k_1,k_2,k_3)$ to the following equation
	\[
	k_1\log a\log b+k_2\log a\log c+k_3\log b\log c=0
	\]
	must have $k_1=0$,  for otherwise
	\[
	\frac{\log a\log b}{\log c}+\frac{k_2}{k_1}\log a+\frac{k_3}{k_1}\log b=0
	\]
	and so $\exp(\log a\log b/\log c)=a^{k_2/k_1}b^{k_3/k_1}$ which is algebraic. However, if $k_1=0,$ then $k_2=k_3=0$ or else $\log b/\log a\in\mathbb{Q}.$ Hence
	\[
	1,\frac{\log a}{\log c}, \frac{\log a}{\log b}
	\]
	are $\mathbb{Q}$-linearly independent. Therefore, Theorem \ref{Six} implies that at least one of the triples 
	\[
	(3,5,7), (3,5,11),(3,5,13),
	\]
	say $(a,b,c)$, is such that
	\[
	1,\frac{\log a}{\log b},\frac{\log a}{\log c}
	\]
	is $\mathbb{Q}$-linearly independent. This proves Theorem \ref{three choices}.\end{proof}
	\section{Digit-special numbers}
	
	It is natural to begin with the more general case of digit-special numbers, and then specialise to the settings of Theorem \ref{A146025} and Theorem \ref{A230360}. As such, in this section we present the proof of Theorem \ref{MAIN}, beginning with two lemmas that develop the majority of the new machinery required. In what follows, we say that $p_1,\dots,p_k$ are strongly multiplicatively independent if $$1,\log p_1/\log p_2,\dots,\log p_1/\log p_k$$ are linearly independent over the field of rational numbers. From Lemma \ref{Sch} and assuming Schanuel's conjecture, this is the case when
	\[
	1,\log p_1,\dots,\log p_k
	\]
	are $\mathbb{Q}$-linearly independent. For $k=2,$ the condition is simply saying that $\log p_1/\log p_2\notin\mathbb{Q}.$
	\begin{lemma}\label{Thfraction}
		Let $k\geq 2$ be an integer and $p_1,\dots, p_k$ be strongly multiplicatively independent integers and for each $i\in\{1,\dots,k\}$ let $a_i\in\{0,\dots,p_i-1\}$. If $$\sum_{i=1}^k \frac{\log (p_i-1)}{\log p_i} <k-1,$$ then the set of numbers in $[0,1]$ whose $p_i$-ary expansion does not contain the digit $a_i$ for all $i\in\{1,\dots,k\}$ has Hausdorff dimension zero.
	\end{lemma}
	
	\subsection{Proof of Lemma \ref{Thfraction}}
	Let $p \in \N$ and $a\in\{0,\dots,p-1\}.$ Define
	\[
	A_p(a)=\overline{\{x\in [0,1]: \text{the $p$-ary expansion of $n$ does not contain the digit $a$}\}},
	\]
	adopting the convention that whenever possible a number $x$ should be written with a terminating digit expansion. A simple calculation shows $\ubox A_p(a)=\log (p-1)/\log p$, (see \cite[Section 1.3]{BP}, \cite[Chapter 4]{Fa} for further details). Hence, if $\sum_{i=1}^k \log (p_i-1)/\log p_i<k-1$, then 
	\[
	\Haus A_{p_1}(a_1)\cap\dots \cap A_{p_k}(a_k)=\boxd A_{p_1}(a_1)\cap\dots \cap A_{p_k}(a_k)=0
	\]
	by Section \ref{UniSmall}. \hfill $\square$
	
	\begin{lemma}\label{Thinteger}
		Let $k\geq 2$ and $p_1,\dots, p_k$ be strongly multiplicatively independent numbers and for each $i\in\{1,\dots,k\},$ let $a_i\in\{0,\dots,p_i-1\}.$ If $$\sum_{i=1}^k \log (p_i-1)/\log p_i<k-1,$$ then for each $\epsilon>0$ there exists a constant $C>0$ such that for each $N\geq 1$,
		$$
		\#\{m\in \{0,\dots, N\} : \textnormal{the $p_i$-ary expansion of $m$ does not contain $a_i$ for all $i=1,\dots,k$}\} \leq CN^\epsilon.
		$$
		
		Moreover, for $$A(a_1,\dots, a_k) = \{m \in \N : \textnormal{the $p_i$-ary expansion of $m$ does not contain $a_i$ for all $i=1,\dots,k$}\},$$ we have
		\[
		\overline{d}\left(\{n\in\mathbb{N}:  A(a_1,\dots,a_k)\cap [p_1^{n},p_1^{n+1}]\neq\emptyset         \}\right) = 0.
		\]
	\end{lemma}
	
	\subsection{Proof of Lemma \ref{Thinteger}}
	Let $p \in \N$ and for each $a\in\{0,\dots,,p-1\}$ define
	\[
	A_p(a)=\{n\in\mathbb{N}: \text{the $p$-ary expansion of $n$ does not contain the digit $a$}\}.
	\]
	Let $k \in \N$, $p_1<\dots<p_k$ be strongly multiplicatively independent integers and $(a_1,\dots,a_k)$ be an arbitrary $k$-tuple with $a_i\in\{0,\dots,p_i-1\}$ for each $i\in\{1,\dots,k\}$. For brevity, we assume $a_1=a_2=\dots=a_k=0$ and note that all the other cases can be treated similarly. Thus, henceforth we write $A_{p_i}$ for $A_{p_i}(0)$. Define
	\[
	K=A_{p_1}\times\dots\times A_{p_k}\subset\mathbb{N}^k.
	\]
	We are interested in the intersection $l_K=K\cap l$, where $l$ is the diagonal line \[l=\{(n,\dots,n): n\in\mathbb{N} \}.\] First, for each $(n,\dots,n)\in l_K$ we wish to find a suitable way to renormalize it.  To do this, we define the vector (in what follows, $\{.\}$ is the fractional part function),
	\[
	\textbf{a}_n=n/p_1^m  \left(1, p_2^{\{m\log p_1/\log p_2\}},\dots,p_k^{\{m\log p_1/\log p_k\}}\right).
	\]
	By construction, $\textbf{a}_n$ is contained in
	\[
	A_{p_1}/p_1^m\times A_{p_2}/p_2^{m_2}\times \dots A_{p_k}/p_k^{m_k}\subset [1,p_1]\times [1,p_1p_2]\times\dots\times [1,p_1p_k]
	\]
	for suitable integers $m_2,m_3,\dots, m_k.$ For each $i\in\{1,\dots,k\}$, we see that
	\[
	A_{p_i}/p_i^{m_i}\subset \overline{\{x\in [1,p_1p_i]: \text{The $p_i$-ary expansion of $x$ does not have digit $0$}               \}}:=B_{p_i}.
	\]
	Observe that $B_{p_i}$ is a subset of a scaled version of a closed $\times p_i\mod 1$ invariant set with Hausdorff dimension $\log (p_i-1)/\log p_i$. Indeed for each $i,$, we first consider the following set
	\[
	B'_{p_i}=\overline{\{x\in [0,1]: \text{The $p_i$-ary expansion of $x$ does not have digit $0$}\}}.
	\]
To see how $B'_{p_i}$ is $\times p_i \mod  1$ invariant, consider the following construction, which closely mirrors the construction of the middle-third Cantor set. We start with the unit interval $[0,1]$, then decompose it equally into $p$ pieces, each with length $1/p.$ We now cut out the first interval, $[0,1/p).$ Then, inside each interval $[j/p,(j+1)/p),j\in\{1,2,\dots,p-1\}$ we cut out the first $1/p$ portion, that is, $[j/p,j/p+1/p^2).$  In this way, we obtain a decreasing sequence of compact sets which converge to $B'_{p_i}.$ Clearly, this set $B'_{p_i}$ is closed and $\times p_i\mod 1$ invariant. After constructing the set $B'_{p_i},$ we consider the scaled set $p^{k_i}_iB'_{p_i},$ where $k_i$ is the smallest integer with $p_i^{k_i}>p_1p_i.$ For each integer $j\in\{1,\dots,p^{k_i}_i\},$ the set $\overline{p^{k_i}_iB'_{p_i}\cap (j,j+1)}$ is empty, or else it is the translated set $B'_{p_i}+j.$ \\
	
As $n$ varies in $[p^m_1, p^{m+1}_1)$, the vectors $\textbf{a}_n$ are contained in a line through the origin with direction vector \[(1, p_2^{\{m\log p_1/\log p_2\}},\dots,p_k^{\{m\log p_1/\log p_k\}}).\tag{Direction}\] Denoting this line as $l_m$, we see that all values of $\textbf{a}_n$ (if they exist) must be contained in 
	\[
	l_m\cap (B_{p_1}\times\dots\times B_{p_k}).
	\]
	Consider the intervals $[n,n+1]$ for $n\in \{p^{m}_1,\dots,p^{m+1}_1-1\}.$ Thus, any $\textbf{a}_n$ (if it exists) must have a first coordinate in the interval $[n/p^m_1,(n+1)/p^m_1].$ We decompose $l_m$ into closed line segments of equal length and disjoint interiors according to the first coordinate, i.e. the components have a first coordinate of form $[j/p^m_1,(j+1)/p^m_1]$ for integers $j.$ We denote this collection of line segments $\mathcal{I}_m$, and wish to estimate the length of those line segments. We know the length of the projection of the first coordinate, say, $d>0.$ We also know the direction vector of the line $l_m$, say $\mathbf{t}=(t_1,t_2,\dots,t_k).$ Then, the length of the line segments will be equal to
	\[
	\frac{\sqrt{t^2_1+\dots+t^2_k}}{|t_1|} d.
	\] 
	Together with (Direction), we see that the length we want to compute is in the range $$\left[p^{-m}_1,p_1^{-m}\sqrt{1+p^2_2+\dots p^2_k}\right].$$ By Section \ref{UniSmall}, we see that for each $\epsilon>0$, there is an integer $N_\epsilon>0$ such that for each $m\geq N_\epsilon$, the number of elements in $\mathcal{I}_m$ intersecting $B_{p_1}\times\dots\times B_{p_k}$ is smaller than
	$
	p^{\epsilon m}_1.
	$ Therefore, for $m\geq N_\epsilon$, the number of points $(n,\dots,n)$ on $l_K$ with $n\in [p^{m}_1,p^{m+1}_1)$ is at most $p_1^{\epsilon m}.$ Thus, there is a constant $C>0$ such that for all $N\geq 1,$
	\[
	\#A_{p_1}(0)\cap \dots\cap A_{p_k}(0)\cap [1,N]\leq C N^\epsilon.
	\]
	This concludes the first part. For the second, we utilise Section \ref{Eqularge}.\\
	
	Suppose that $\textbf{a}_n$ exists for some $n\in [p^m_1,p^{m+1}_1)$ for all $m \in \mathcal{M} \subseteq \N$, where $\overline{d}(\mathcal{M}) > 0$. This implies that
	\[
	\overline{\{(1, p_2^{\{m\log p_1/\log p_2\}},\dots,p_k^{\{m\log p_1/\log p_k\}})\}_{m\in\mathcal{M}}}
	\]
	has positive Lebesgue measure, forcing
	\[
	B_{p_1}\times\dots\times B_{p_k}
	\]
	to have dimension at least $k-1$, by Section \ref{DD}. This is a contradiction and concludes the proof of the second part. \hfill $\square$
	
	\subsection{Proof of Theorem \ref{MAIN}}
	Let $p_1,p_2,\dots=3,4,5,\dots$ be the list of prime numbers greater than two together with $4$. Under Schanuel's conjecture and Lemma \ref{Sch}, we see that $p_1,\dots$ are strongly multiplicatively independent. Observe that for each $i\geq 1,$
	\[
	\frac{\log (p_i-1)}{\log p_i}=1+\frac{1}{\log p_i} \log (1-p_i^{-1})\leq 1-\frac{1}{p_i\log p_i}.
	\]
	It may be easily numerically computed that the following convergent sum
	\[
	C=\sum_{i\geq 1} \frac{1}{p_i \log p_i} \approx 1.09561.
	\]
	and, in fact
	\[
	C>\sum_{i=1}^{26} \frac{1}{p_i \log p_i}\geq 1.00112>1.
	\]
	Note that $p_{26}=101$, the $26$-th prime number. Next, we apply Lemma \ref{Thinteger} with $k=26$ and a collection $a_1,\dots,a_k$ chosen arbitrarily. Fix a small number $\epsilon>0$, for each such collection of $a_1, \dots, a_k$, there is a constant $C_{a_1,\dots,a_k}$ such that among the first $N$ integers, all but at most $C_{a_1,\dots,a_k} N^\epsilon$ many of them contain $a_i$ in their $p_i$-ary expansion for at least one $i\in\{1,\dots,k\}.$ There are finitely many choices of the tuple $a_1,\dots, a_k$, and thus setting
	$$
	C :=\sum_{a_1,\dots,a_k} C_{a_1,\dots,a_k}<\infty
	$$
	completes the first part of the proof. For numbers in $(0,1)$ we may argue similarly and apply Lemma \ref{Thfraction}. \hfill $\square$
	\section{Numbers with only binary digits in different bases}
	In this Section we prove Theorem \ref{A146025} and Theorem \ref{A230360}.
	\subsection{Proof of Theorem \ref{A146025}}
	We utilise the general strategy found in the proof of Lemma  \ref{Thinteger}. Note that, for a base $b$, the set $A_b$ of numbers in $[0,1]$ whose $b$-ary expansion contain only the digits $0,1$ has Hausdorff dimension $\log 2/\log b$. Thus, Theorem \ref{A146025} follows by a direct modification of the proof of Lemma \ref{Thinteger} (with $k=2$ in the statement) together with the fact that
	\[
	\frac{\log 2}{\log 4}+\frac{\log 2}{\log 5}\approx 0.930677<1.
	\]
	\hfill $\square$
	\subsection{Proof of Theorem \ref{A230360}}
	First, observe
	\[
	\frac{\log 2}{\log 3}+\frac{\log 2}{\log 4}>1=2-1.
	\]
	Hence, we may not proceed as before by utilising the method of Lemma \ref{Thinteger}.\\
	
By a result in \cite{Wu} and \cite{Sh},  for any line $l$ not parallel with the coordinate axes, the intersection $l\cap A_4\times A_3$ has dimension at most $\log 2/\log 3+\log 2/\log 4-1$. This is the reason for the exponent that appeared in Theorem \ref{A230360}. By \cite[Lemma 11.1]{Y19}, there is an integer $M>0$ and closed $\times 3^M\mod 1$ invariant sets $A'_3, A''_3\subset [0,1]$ with $$A_3\subset A'_3+A''_3,$$ $$|\Haus A'_3-0.49|<0.0001,$$ and $$\Haus A'_3+\Haus A''_3\leq \log 2/\log 3+0.0001.$$ By applying the argument in the proof of Lemma \ref{Thinteger}, for each integer $n\geq 1$, we may map $\{x=y\}\cap \{x\in [4^{n},4^{(n+1)}]\}$ to a line passing through the origin with slope $9^{{\{n\log 4/\log 9\}}}$. Denote this line $l_n.$ We wish to estimate how large $l_n\cap A_4\times A_3\cap \{x\in [1,9]\}$ can be. Considering $l_n\cap A_4\times (A'_3+A''_3)$, we note this can be written as
	\[
	\bigcup_{t\in A''_3} ((l_n-(0,t))\cap A_4\times A'_3)+t.
	\]
	In general, $(l_n-(0,t))\cap A_4\times A'_3$ is small for each individual $t$, however, since there are uncountably many elements in $A''_3$, we cannot say anything about the union. However, in our discrete case we may bypass this issue.\\
	
	Our aim is to decompose $A_3$ in such a way that we can utilise the uniform small dimension result discussed in Section \ref{UniSmall}. Let $m,r$ be two integers with $m\geq 2$ and $r\in\{0,\dots,m-1\}$. Denote $A_3(m,r)$ to be the subset of $A_3$ consisting those numbers whose ternary expansion may only have the digit $1$ in the $(km+r)$-th positions for all integers $k\geq 0$. We then observe
	\[
	A_3=\sum_{r=0}^{m-1} A_3(m,r)
	\]
	as a sumset. Choosing an integer $s<m-1$ such that
	\[
	A_3'=\sum_{r=0}^{s} A_3(m,r)
	\]
	and 
	\[
	A_3''=\sum_{r=s+1}^{m-1}A_3(m,r)
	\]
	yields $A_3=A_3'+A_3''.$ By choosing $m$ to be suitably large as well as $s$ we may force
	\[
	0.49<\Haus A_3'<0.5.
	\]
	Hence $\Haus A_3'+\Haus A_4<1$ and $A_3'$ is $\times 3^m \mod 1$ invariant. \\
	
	Recall that we wish to investigate $l_n\cap 4A_4\times 27A_3$. In particular, we wish to count the number of points in $l_n\cap 4A_4\times 27A_3$ whose $x$ coordinate is of form $k/4^n, k\in\{4^n,\dots,4^{n+1}-1\}.$ If $a\in l_n\cap 4A_4\times 27A_3$, then there is a $t\in 27A''_3$ such that $(a-(0,t))\in 4A_4\times 27A_3'$. It is easy to check that we only need to consider those $t$ with a terminating ternary expansion of at most $[(n+1)\log 4/\log 3]$ many digits in total.\\
	
	Let $\epsilon>0$ be a small number. There exists an integer $N_\epsilon$ such that we may apply Section \ref{UniSmall}. Assume that $n\geq N_\epsilon$. For each $t$ as above, we see that $(l_n-(0,t))\cap 4A_4\times 27A_3$ can be covered by at most $4^{n\epsilon}$ balls of radius $4^{-n}$. Moreover, there is a constant $C$ (depending only on $m$) such that there are no more than $C4^{n(\log2/\log 3-0.49)}$ many such $t$ to be considered. Hence, $l_n\cap 4A_4\times 27A_3$ can be covered by at most $C4^{n(\log2/\log 3-0.49+\epsilon)}$ many balls of radius $4^{-n}.$ This implies that among all integers in $[4^n,4^{n+1}),$ there are no more than $C4^{n(\log 2/\log 3-0.49+\epsilon)}$ many of them with base $3$ and $4$ expansions containing only binary digits. We can replace $0.49$ to be any number smaller than $0.5$ (by choosing $m$ to be large enough), concluding the proof of first part of Theorem \ref{A230360}.\\

	\begin{figure}[h]
		\includegraphics[width=0.75\linewidth, height=8cm]{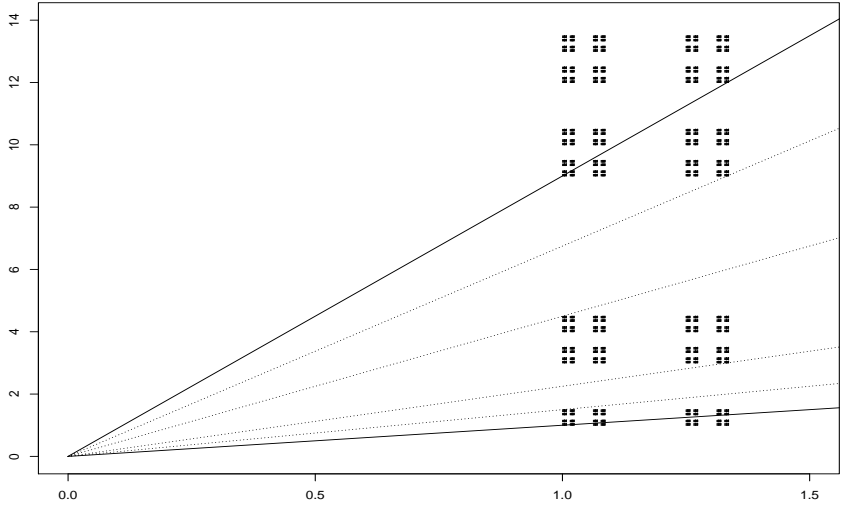} 
		\caption{The two solid lines have slopes $1$ and $9$. The four dashed lines have slopes $1.5,2.25,4.5$ and $6.25.$}
		\label{fig:figure 2}
	\end{figure}
	
	For the second part, note that $l_n$ is a line passing through the origin with slope $9^{\{n\log 4/\log 9\}}$. As $n$ varies through the natural numbers, $9^{\{n\log 4/\log 9\}}$ will take values in $[1,9].$ Figure \ref{fig:figure 2} illustrates that there are regions of slopes such that the lines passing through the origin with those slopes cannot intersect \[(4A_4\cap [1,4])\times (27 A_3\cap [1,27]).\] For example, if \[9^{\{n\log 4/\log 9\}}\in (1.5,2.25)\cup (4.5,6.25),\] then \[l_n\cap (4A_4\cap [1,4])\times (27 A_3\cap [1,27])=\emptyset.\] Since the slopes equidistribute across $[1, 9]$, directly computing the proportion of such regions in $[1, 9]$ shows the above intersection is empty for at least a $0.36907$ portion of $\N$. \hfill $\square$

	\section{Discussion, Conjectures and Open Problems}\label{discuss}
	
	In this section we provide some broad insights into the above topics to help future work. Divided into three subsections, the first deals our numerical analysis on digit special numbers, the second with binary expansions in base $3$ and $4$, and the last with related open problems.
	
	\subsection{Digit-special numbers}
	We have computed the approximate number of digit special numbers  lower than $3^{36}$. In order to check whether $n \in \N$ is special, it suffices to check its base $g$ expansion for $3\leq g\leq b$ and $g$ either a prime number or $4$, where $b$ is the largest prime such that $b^b\leq n$. In the following, we denote this collection of bases $B_n$. If we choose a digit for each base $g$, then the amount of numbers whose base $g$ expansion misses the chosen digit in each base for all $g\in B_n$ is $O(n^{s+\epsilon})$, where 
	\[
	s=\max\left\{0,\sum_{g\in B_n} \frac{\log(g-1)}{\log g}-(\#B_n-1)\right\}.
	\]  
	To understand this choice of $s$ we direct the reader to the hypotheses of Lemma \ref{Thinteger} and the decomposition method found at the beginning of the proof of Theorem \ref{A230360}. As there are $\prod_{g\in B_n} g$ many choices of different possible combinations of missing digits,  a very rough estimate for the amount of special-numbers less than $n$ is

	\[
	\text{Est}(n)=n^{s}\prod_{g\in B_n} g.
	\]
	Letting
	\[
	\text{Real}(n)=\#\{i \leq n : \text{$i$ is digit special}\},
	\]
	in Figure \ref{fig:figure 4} we compare $\text{Est}(n)$ with the actual data, by plotting $G(n)=\log \text{Est}(n)/\log n$ and $R(n)=\log \text{Real}(n)/\log n$. Specifically, it is worth observing that the estimates appear to become quite precise for $n \geq 31$.
	\begin{figure}[h]
		\includegraphics[width=0.75\linewidth, height=7cm]{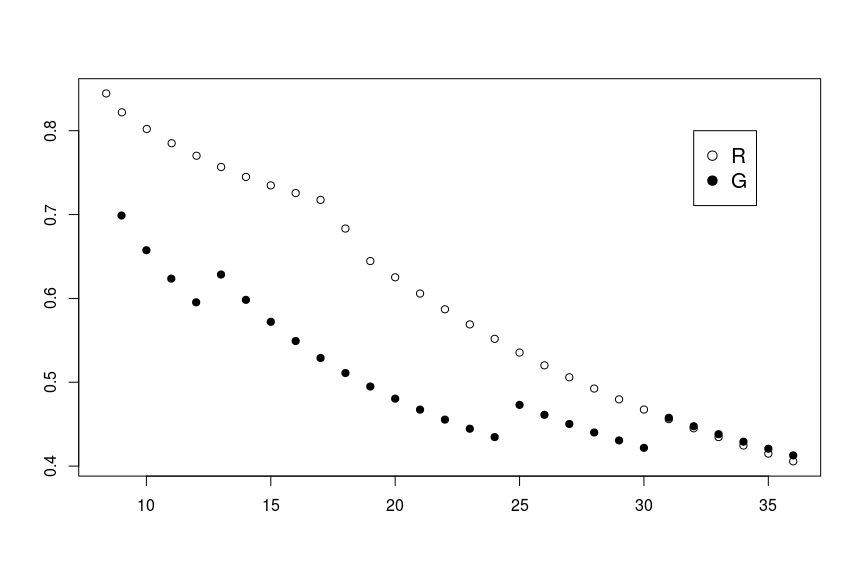} 
		\caption{This is a plot of $R(3^n) , G(3^n)$ for $8\leq n\leq 36$. Corollary \ref{MAIN} implies that the $R$ points will eventually drop near to $0$.}
		\label{fig:figure 4}
	\end{figure}
	
	Although these estimates are somewhat crude, they approximate the true values surprisingly well for large $n$. In order to establish the theoretical reasons for this, further quantitive information on the constant $C$ appearing in Lemma \ref{Thinteger} is required.\\
	
	The overarching message of our analysis is that digit-special numbers are exceptionally rare. Thus, we conclude this part of the discussion with the following two conjectures, which constitute a strengthening of Theorem \ref{MAIN}.
	
	\begin{conj}\label{WHY}
		There are finitely many digit-special integers.
	\end{conj}
	
	\begin{conj}
		Digit-special numbers in $(0,1)$ are rational.
	\end{conj}
	
	\subsection{Binary digit expansions in base $3$ and $4$.}
	Next, we will discuss some further conjectures and questions relating to the sequence \cite{OEIS2} on numbers with only binary digits in their base $3$ and $4$ expansions. For \cite{OEIS2}, it would be interesting to compute the exact density of the appearance of $0$. Recall that in Theorem \ref{A230360} we showed that $0$ must appear at a lower density of at least $0.36907$. In addition, Figure \ref{fig:figure 1} suggests that for the non-zero terms it seems likely the exponent $\log 2/\log 3-0.5$ is essentially sharp. The following questions makes this precise.
	\begin{ques}
		For each $\epsilon>0,$ are there infinitely many integers $n\in\mathbb{N}$ such that $
		\mathcal{S}(n) \geq  4^{n(\log2/\log3-0.5-\epsilon)}$?
	\end{ques}
	It is already interesting to see whether $\mathcal{S}(n)>0$ for infinitely many $n.$ Unfortunately, Theorem \ref{EGRS} cannot help us to find an answer, since in the statement of the theorem we must have $A=B=1$, but then 
	\[
	\frac{1}{2}+\frac{1}{3}<1.
	\] 
	On the other hand, if we were to consider numbers containing $\{0,1,2\}$ in their base $4$ expansion, then we would find infinitely many, since
	\[
	\frac{1}{2}+\frac{2}{3}>1.
	\]
	
	\begin{ques}\label{zeroden}
		What is the lower density of $\{n\in\mathbb{N}: \mathcal{S}(n)=0 \}$?
	\end{ques}
	In relation to Question \ref{zeroden}, note that for the last part of Theorem \ref{A230360}, we were required to identify the proportion of slopes in Figure \ref{fig:figure 2} avoiding a cantor-like set. The estimate given is based on just the largest interval of such slopes. In fact, there are smaller gaps that we did not point out, as illustrated in Figure \ref{fig:figure 3}, which is a zoomed-in picture of Figure \ref{fig:figure 2}. Including these further regions in the calculation yields a small improvement of approximately $0.0115$ to the lower density estimate. Thus, the heart of Question \ref{zeroden} is to compute the sum of the lengths of all such gaps.\\
	
	\begin{figure}[h]
		\includegraphics[width=0.75\linewidth, height=8cm]{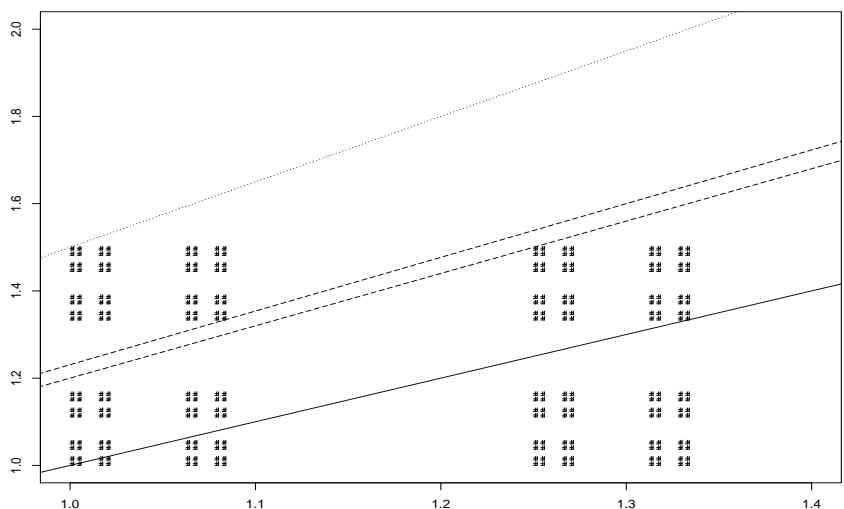} 
		\caption{Zoomed-in version of Figure \ref{fig:figure 2}, the lines have slops $1,1.2,16/13,1.5.$}
		\label{fig:figure 3}
	\end{figure}
	
	As a final question, note that thus far we have separately discussed integers and numbers in $(0,1).$ However, the similarity in the methods used suggests a potential connection, which we describe in the following conjecture.
	
	\begin{conj}\label{STRONG}
		Let $p_1,\dots,p_k$ be strongly multiplicatively independent integers. For each $i\in\{1,\dots,k\}$ and $D_i\subset \{0,\dots,p_i-1\}$, define
		\[
		A_{p_i}(D_i)=\{x\in (0,1): \text{the $p$-ary expansion of $x$ does not have a digit in $D_i$} \}
		\]
		and
		\[
		\tilde{A}_{p_i}(D_i)=\{x\in\mathbb{N}: \text{the $p$-ary expansion of $x$ does not have a digit in $D_i$} \}.
		\]
		
		Furthermore, let $A=\cap_{i=1}^k A_{p_i}(D_i).$ If $$\sum_{i\in\{1,\dots,k\}}\Haus A_{p_i}=s\in (k-1,k),$$ then there exist constants $c$ and $C$ such that
		\[
		c N^{s-(k-1)}\leq \# A\cap [1,N]\leq C N^{s-(k-1)}
		\]
		for all integers $N.$ If $s<k-1$, then $A$ is finite.
	\end{conj}
	We will see shortly that under Schanuel's conjecture, the above conjecture may resolve Graham's problem. 
	\subsection{A return to binomial coefficients}
	Unless otherwise mentioned, Schanuel's conjecture is assumed for the discussions in this subsection. A number of problems related to the prime factors of binomial coefficients $\binom{2n}{n}$ have been discussed, but the alert reader may notice that we actually have not explicitly proved Theorem \ref{PP}. However, one can easily modify  the proof of Theorem \ref{A230360} to show Theorem \ref{PP} with just a few key observations. First, notice that $3,5,7$ satisfy the condition on $p_1,p_2,p_3$ in the result of Section \ref{UniSmall} and that
\begin{equation}\label{calc}
	\frac{\log 2}{\log 3}+\frac{\log 3}{\log 5}+\frac{\log 4}{\log 7}-2\approx 0.0259<0.026.
	\end{equation}
	Without using Schanuel's conjecture, we can use Theorem \ref{three choices} to find an integer $n\in\{7,11,13\}$ such that $1,\log3/\log 5,\log 3/\log n$ are $\mathbb{Q}$-linearly independent. The worst upper bound occurs when $n=13$,
	\[
	\frac{\log 2}{\log 3}+\frac{\log 3}{\log 5}+\frac{\log 7}{\log 13}-2\approx 0.0722<0.073.
	\]
Under Schanuel's conjecture, however, we may set $n=7$ and perform a decomposition as discussed in the proof of Theorem \ref{A230360} together with (\ref{calc}) to deduce Theorem \ref{PP}.\\

Now let us put Conjecture \ref{STRONG} into play. Together with the computations above, we see that Conjecture \ref{STRONG} would imply Conjecture \ref{P} and thus answer the Graham's $\$1000$ problem.\\
	
	Finally, we note how our methods also provide information on natural generalisations of Graham's problem. For example, one may consider $1155=3\times 5\times 7\times 11$ in place of $105,$ see \cite{OEIS8} for some numerical computations. It is conjectured that the only integers $n$ such that $\binom{2n}{n}$ is coprime with $1155$ are $0,1,3160$, motivated by the fact there are no other examples smaller than $10^{10000}.$ Observing that
	\[
	\frac{\log 2}{\log 3}+\frac{\log 3}{\log 5}+\frac{\log 4}{\log 7}+\frac{\log 6}{\log 11}<3,
	\]
	the proof of Theorem \ref{A146025} may be generalised to yield the following.
	\begin{thm}
		Assume Schanuel's conjecture. Let $N$ be an integer. Denote $G(N)$ to be the number of positive integers $n\leq N$ such that $\binom{2n}{n}$ is coprime with $1155$. Then for all $\epsilon>0,$ we have
		\[
		G(N)\leq N^{\epsilon}
		\]
		for all large enough $N.$
	\end{thm}
	\subsection{Strongly multiplicative independence}
	If $k\geq 3$, we have seen that Schanuel's conjecture implies some integers $p_1,p_2,\dots,p_k$  are strongly multiplicatively independent if they are multiplicatively independent. In other words, if
	\[
	1,\log p_1,\dots, \log p_k
	\]
	are $\mathbb{Q}$-linearly independent. Without assuming Schanuel's conjecture, we saw that at least one of the triples $(3,5,7), (3,5,11), (3,5,13)$ are strongly multiplicatively independent. Thus there exists at least one strongly multiplicatively independent integer triple.\\
	
If one wishes to consider strong multiplicative independence of quadruples, the situation is far more complex and the proof of Theorem \ref{three choices} cannot be directly generalized. It is unknown whether such quadruples exist, although, under Schanuel's conjecture, one would expect there to be multitudes. 

	\section{Acknowledgement}
	The authors would like to thank Douglas Howroyd for many useful discussions and help with the numerical computations used to create Figure \ref{fig:figure 4}. The authors also want to thank Carlo Sanna for bringing us the connection between Graham's problem and results in an early version of this manuscript. SAB was supported by a \emph{Carnegie Trust PhD Scholarship} (PHD060287) and HY was financially supported by the University of St Andrews,  the University of Cambridge and the Corpus Christi College, Cambridge. HY has received funding from the European Research Council (ERC) under the European Union’s Horizon 2020 research and innovation programme (grant agreement No. 803711).
	
	\bibliographystyle{amsplain}
	
\end{document}